\documentclass[a4paper,12pt]{amsart}
\usepackage{amssymb,amsmath,amsfonts,a4wide, color, mathtools}
\usepackage{url}
\usepackage{relsize}
\usepackage{xcolor}
\usepackage{hyperref}
\hypersetup{
    colorlinks=true, 
    linktoc=all,     
   linkcolor=blue,  
}

\newtheorem{pro}{Proposition}[section]
\newtheorem{teo}[pro]{Theorem}
\newtheorem{lema}[pro]{Lemma}
\newtheorem{coro}[pro]{Corollary}
\theoremstyle{definition}
\newtheorem{defi}[pro]{Definition}
\newtheorem{rema}[pro]{Remark}

\numberwithin{equation}{section}

\def \K{\mathcal{K}}
\def \F{\mathcal{F}}

\def \T{\overline{T}}

\newcommand{\w}{\wedge}
\DeclareMathOperator{\li}{\mathcal{L}}

\newcommand{\tr}{\textsl{tr}\,}
\newcommand{\bb}{\mathbb}
\DeclareMathOperator{\Ric}{\textsl{Ric}\,}

\def \E {\mathcal{E}}

\def \T{T^{\prime}}
\def \i {\lrcorner}
\def \g{\mathfrak{g}}

\newcommand{\End}{\operatorname{End}}
\def\Ric{\mathrm{Ric}}

\newcommand{\lie}[1]{\mathfrak{#1}}
\newcommand{\Lie}[1]{\textsl{#1}}

\DeclareMathOperator{\SU}{\Lie{SU}}
\DeclareMathOperator{\U}{\Lie{U}}
\DeclareMathOperator{\su}{\lie{su}}

\DeclareMathOperator{\aut}{\lie{aut}}

\DeclareMathOperator{\m}{\lie{m}}

\newcommand{\h}{\lie{h}}

\DeclareMathOperator{\di}{d}
\DeclareMathOperator{\Ker}{Ker}
\DeclareMathOperator{\im}{Im}

\renewcommand{\Re}{\textsl{Re\,}}

\DeclareMathOperator{\vol}{vol}

\newcommand{\Kk}{\mathcal{K}}

\newcommand{\ba}{\begin{array}}\newcommand{\ea}{\end{array}}

\renewcommand{\&}{{\footnotesize \&}}

\def\beq{\begin{equation}}
\def\eeq{\end{equation}}
\def\bea{\begin{eqnarray*}}
\def\eea{\end{eqnarray*}}
\def\beaa{\begin{eqnarray}}
\def\eeaa{\end{eqnarray}}
\def\ba{\begin{array}}
\def\ea{\end{array}}
\def\Sym{\mathrm{Sym}}
\def\T{\mathrm{T}}
\def\id{\mathrm{id}}
\def\scal{\mathrm{scal}}
\def\la{\langle}
\def\ra{\rangle}
\def\Cc{\mathcal{C}}
\def\spann{\mathrm{span}}
\def\Hom{\mathrm{Hom}}
\def\tr{\mathrm{tr}}

\def\sideremark#1{\ifvmode\leavevmode\fi\vadjust{\vbox to0pt{\vss
 \hbox to 0pt{\hskip\hsize\hskip1em
 \vbox{\hsize2.5cm\tiny\raggedright\pretolerance10000
 \noindent #1\hfill}\hss}\vbox to8pt{\vfil}\vss}}}%

\begin{document} 
\title[]{Deformations of nearly $G_2$ structures}
\author[P.-A. Nagy]{Paul-Andi Nagy} 
\address[Paul-Andi Nagy]{Institut f\"ur Geometrie und Topologie, Fachbereich Mathematik, Universit\"at Stuttgart, Pfaffenwaldring 57, 70569 Stuttgart, Germany}
\email{paul-andi.nagy@mathematik.uni-stuttgart.de}
\author[U. Semmelmann]{Uwe Semmelmann}
\address[Uwe Semmelmann]{Institut f\"ur Geometrie und Topologie, Fachbereich Mathematik, Universit\"at Stuttgart, Pfaffenwaldring 57, 70569 Stuttgart, Germany}
\email{Uwe.Semmelmann@mathematik.uni-stuttgart.de}
\date{\today} 
\begin{abstract}
We describe the second order obstruction to deformation for nearly $G_2$ structures on compact manifolds. Building on work of B.Alexandrov and U.Semmelmann this allows proving rigidity under deformation for the proper nearly $G_2$ structure on the Aloff-Wallach space $N(1,1)$.
\end{abstract}
\maketitle
\section{Introduction}
Consider a compact oriented manifold $(M^7,\vol)$. A $G_2$ structure on $M$ is a $3$-form $\varphi$ on $M$ which is 
stable in the sense of \cite{Hi} and compatible with the orientation choice. Such a structure induces in an unique way a Riemannian 
metric $g_{\varphi}$ on $M$, with respect to which we consider $\psi:=\star_{g_{\varphi}}\varphi$. The $G_2$ structure is called (strictly) 
nearly $G_2$ provided that 
\begin{equation} \label{str-eqn}
\di\!\varphi=\tau_0 \psi
\end{equation}
for some 
$\tau_0 \in \mathbb{R}^{\times}.$  It is a well established result \cite{BFK} that nearly $G_2$ structures are in $1:1$ correspondence with Riemannian metrics in dimension $7$ admitting Killing spinors. In particular $g_{\varphi}$ is an Einstein metric of positive scalar 
curvature, a fact which further drives the research in this area. The nearly $G_2$ structure is called {\it{proper}} if
$\aut(M,g_{\varphi})\subseteq \aut(M,\varphi)$; equivalently $g_{\varphi}$ is required to admit a one-dimensional space of Killing spinors. 
The main classes of examples known are 
\begin{itemize}
\item[$\bullet$] homogeneous, as classified in \cite{FKMS}, including the Aloff-Wallach spaces $N(k,l)$
\item[$\bullet$] obtained from a  canonical variation \cite{FKMS}, \cite{GS} of a $3$-Sasaki metric in dimension $7$. 
\end{itemize}
A distinguished r\^ole is played by the Aloff-Wallach space $N(1,1)$ which supports   a $3$-Sasaki metric and two nearly-$G_2$ structures  of proper type \cite{Ca}. 

The deformation theory of (proper) nearly $G_2$ structures, which is a potential tool for producing new examples is the main focus 
in this paper. Some evidence in this direction is supported by the fact that $3$-Sasaki metrics in dimension $7$ containing $\mathbb{T}^3$ in their automorphism group can be deformed to Sasaki-Einstein structures \cite{coev2}. According to \cite{AlS}, infinitesimal deformations of nearly $G_2$ structures correspond to the kernel $\F_4$ of $\Delta^{g_{\varphi}}-\tau_0^2$ acting on $\Omega^4_{27}(M,g_{\varphi}) \cap 
\ker \di$. Those are actually deformations which are normalised to lie, up to the action of the diffeomorphism group, in the Ebin slice for Riemannian metrics on $M$. We consider the cubic polynomial 
\begin{equation} \label{K1}
\mathbb{K}:\mathcal{F}_4 \to \mathcal{F}^*_4, \ \mathbb{K}(\alpha) \gamma :=\int_M Q_2(\alpha) \wedge \gamma.  
\end{equation}
Here $Q_2$ is the quadratic form associated to an explicit bilinear form $b_2:\Lambda^4_{27} \times \Lambda^4_{27} \to\Lambda^3_{27}$ between $G_2$-representation spaces.

The main result result of this paper is the following 
\begin{teo} \label{main}
Let $(M^7,\varphi)$ be a compact manifold equipped with a proper nearly $G_2$ structure. The set of infinitesimal deformations which are unobstructed to second order is parametrised by $\mathbb{K}^{-1}(0)$.
\end{teo}
Here elements $\alpha \in \mathcal{F}_4$ are called unobstructed to second order provided they arise from a second order Taylor series 
of nearly $G_2$ structures. This notion models, at order two, instances of deformation by smooth curves of nearly $G_2$ structures. As it is well-known those can be  parametrised  algebraically by using the    linearisation $\alpha \mapsto \widehat{\alpha}$  of Hitchin's duality map \cite{Hi}. Differentiating the structure equations \eqref{str-eqn} reveals that 
deformation theory is governed by the first order differential operator $D:\Omega^4(M) \to \Omega^4(M), D\alpha=\di\widehat{\alpha}-\tau_0\alpha$. This is subelliptic, thus has a well defined Hodge theory which shows that (unnormalised)  infinitesimal deformations lie in $\ker(D)$ and 
also that $\alpha \in \ker(D)$ is unobstructed to second order provided that $Q_2(\alpha)$ is $L^2$-orthogonal to 
$\di^{\star} \ker(D^{\star})$.

With these preliminary observations in hand, the proof of Theorem \ref{main} consists in examining how $D$ interacts with the type decomposition of 
$\Omega^4(M)$. Here, a crucial observation we make is that a small modification of the Ebin slice for Riemannian metrics produces 
invariant subspaces for $D$ and $D^{\star}$. This allows breaking those operators into blocks with very simple Hodge theory.
 
For the Aloff-Wallach space $N(1,1)$ equipped with the canonical variation of the $3$-Sasaki metric, we show that Theorem \ref{main} provides a computationally  efficient way to describe second order deformations.
\begin{teo} \label{main1}
All infinitesimal deformations of $N(1,1)$ are obstructed to second order.
\end{teo}
The first ingredient used for the proof is the representation theoretic description of $\mathcal{F}_4$ which turns out to 
be isomorphic to $\su(3)$ as an $SU(3)$-representation (cf. \cite{AlS}). The second is the explicit computation of the bilinear form $b_2$ restricted to $\F_4$. This is performed in detail in the last section of the paper.

Related results have been proved by L.Foscolo in \cite{Fos} for nearly K\"ahler manifolds in dimension $6$ based on our earlier work \cite{MNS}, \cite{MS}. In particular he was able to show that the nearly K\"ahler metric on the flag manifold $F(1,2)$ has no non-trivial deformations. As it is the case in this paper, his work relies on the explicit parametrisation of curves of $\SU(3)$-structures given by the linearisation of the duality map.  
\footnote{
Shortly after we have posted the first version of this paper on the arXiv, S.Dwivedi and R.Singhal have announced related results on nearly $G_2$ deformations. Their arguments are rather involved and seem to be partially incomplete.}
\section{Preliminaries} \label{prel}
\subsection{Linear algebra} \label{alg}
Consider a nearly $G_2$ manifold $(M,\varphi,\vol)$ as above; the induced metric will be simply denoted by $g$ in what follows. The metric $g$ is recovered from the pair of $G_2$-data $(\varphi,\vol)$ according to  
$$(v \lrcorner \varphi) \w (w \lrcorner \varphi) \w \varphi=-6g(v,w)\vol.$$
The Hodge star operator constructed from $(g,\vol)$ will be denoted by $\star$. In many algebraic computations it is useful to express 
$\varphi$ and $\psi:=\star \varphi$ in an adapted frame. This is a local orthonormal basis $\{e^k,1 \leq k \leq 7\}$ of one forms w.r.t. which 
\begin{equation*}
\begin{split}
\varphi=&e^{123}+e^{145}-e^{167}+e^{246}+e^{257}+e^{347}-e^{356}\\
\psi=&e^{4567}-e^{1247}+e^{1256}-e^{2345}+e^{2367}-e^{3146}-e^{3157}.
\end{split}
\end{equation*}
As $G_2$-representations the spaces of $k$-forms $\Lambda^k$ with $k=3,4$ split into irreducible components as $\Lambda^k=\Lambda_1^k \oplus \Lambda_7^k \oplus \Lambda_{27}^k$. Accordingly, we write $\alpha= \alpha_1 + \alpha_7 + \alpha_{27}$ whenever $\alpha \in \Lambda^k$. Note that $\Lambda^4_1=\bb{R}\psi$ and $\Lambda^3_1=\mathbb{R}\varphi$. Similarly the space of $2$-forms splits as $\Lambda^2 = \Lambda^2_{7} \oplus \Lambda^2_{14}$ as a $G_2$-representation.
 
Throughout this paper we use the metric to identify tangent vectors and $1$-forms, as well as endomorphisms and $(2,0)$-tensors via 
$A \in \End \T \mapsto g(A \cdot ,\cdot)$. 
Consider the action 
$(A,\alpha) \in \End \T  \times \Lambda^{\star} $ given by  $A_{\star}\alpha= Ae_i \w (e_i \lrcorner \alpha)$, where $\{e_i\}$ is
some orthonormal basis  of $\T$. Note that we use here and in the following  the Einstein summation convention and sum over repeated indices. This action allows to define a $G_2$-invariant  linear 
iso\-morphism $i:\Sym^2_0\T \to \Lambda^3_{27}$ via $S \mapsto S_{\star}\varphi$. A few well-known identities to be repeatedly 
used in this paper are 
\begin{equation} \label{idg21}
\star (S_{\star}\psi) = - S_{\star}\varphi
\qquad \mbox{and} \qquad
 \vert i(S) \vert^2=2\vert S \vert^2
\end{equation}
as well as 
\begin{equation}\label{s-star}
i(S) \w (v_1 \lrcorner \psi ) \w v_2 = 2  g(Sv_1, v_2) \vol
\end{equation}
whenever $S\in \Sym^2_0\T$ and $v_1,v_2 \in \T$. All these facts can either be proved by direct computation in an adapted frame 
or looked up in \cite{Br,CI}. Note that Bryant's orientation convention in \cite{Br} is opposite to ours and that the isomorphism $i$ differs from his by a factor of $\tfrac 12$. The inner product on $\Sym^2\T$ we work with here is $\la S_1, S_2\ra = \tr(S_1 S_2)$.

\medskip

Two specific algebraic facts will be needed in this work. The first is the following 
\begin{lema} \label{hat}
We have a linear map $\Lambda^4 \mapsto \Lambda^3, \ \alpha \mapsto \widehat{\alpha}$ uniquely determined from having 
\begin{equation} \label{hat2}
\widehat{\alpha} \w (v \lrcorner \psi)+\varphi \w (v \lrcorner \alpha)=0 
\end{equation}
for all $v \in \T$.
Moreover 
\begin{equation} \label{hat3}
\widehat{\alpha}=-\star \alpha_1+\star \alpha_7 - \star \alpha_{27}. 
\end{equation}
\end{lema}
This is checked by direct computation in an adapted frame, using sample elements. 
Secondly, we need to deal with the natural $G_2$-invariant polynomial introduced below.

\begin{pro} \label{pol}There exists a symmetric bilinear form
$ b_2:\Lambda^4  \times \Lambda^4  \to \Lambda^3 $ 
uniquely determined from having 
\begin{equation}\label{q2}
b_2(\alpha_1,\alpha_2) \wedge (v \lrcorner \psi)+\widehat{\alpha}_1 \wedge (v \lrcorner \alpha_2)+\widehat{\alpha}_2 \wedge (v \lrcorner \alpha_1)=0
\end{equation}
for all $v \in \T$. Letting $Q_2(\alpha) = b_2(\alpha, \alpha)$ be the quadratic form associated to $b_2$, then

\begin{itemize}
\item[(i)]$Q_2(\Lambda^4_{27}) \subseteq \Lambda^3_1 \oplus \Lambda^3_{27}$  and
\medskip
\item[(ii)] the cubic polynomial $Q:\Lambda^4_{27} \to \mathbb{R}$  given  by $Q(\alpha)\vol:=Q_2(\alpha) \wedge \alpha$ satisfies 
\begin{equation*}
Q(\alpha)=-2\langle q(\alpha,\alpha), i^{-1}(\star \alpha) \rangle
\end{equation*}
where  the symmetric bilinear form $ q:\Lambda_{27}^4 \times \Lambda^4_{27} \to \Sym^2\T$ is defined by the equation 
$ \ q(\alpha,\alpha)(v_1,v_2):=\langle v_1 \lrcorner \alpha, v_2 \lrcorner 
\alpha \rangle$.
\end{itemize}
\end{pro}
\begin{proof}All statements follow from the following simple observation. Pick $\alpha \in \Lambda^4_{27}$ and compute 
\begin{equation*}
\begin{split}
\widehat{\alpha} \w (v_1 \lrcorner \alpha) \w v_2=&-\ast \alpha \w (v_1 \lrcorner \alpha) \w v_2 = -(v_1 \lrcorner \alpha) \w (v_2 \w \ast \alpha) = (v_1 \lrcorner \alpha) \w \ast(v_2 \lrcorner \alpha) \\
=& \quad q(\alpha,\alpha)(v_1, v_2 ) \, \vol.
\end{split}
\end{equation*}
Split $q=q_0 + \tfrac 47 g \otimes \id$ according to $\Sym^2\T=\Sym^2_0\T\oplus \mathbb{R}\id$. Then \eqref{s-star} together with the elementary formula 
$
\varphi \w  (v_1 \lrcorner \psi ) \w v_2 = - 4 g(v_1, v_2)\vol
$ 
imply
\begin{equation*}
(i(q_0(\alpha,\alpha))-\tfrac 27\vert \alpha \vert^2\varphi) \w (v_1 \lrcorner \psi) \w v_2=2q(\alpha,\alpha)(v_1,v_2)\vol.
\end{equation*}
Comparing the last two displayed equations we see that $Q_2(\alpha):=-i(q_0(\alpha,\alpha))+\tfrac 27\vert \alpha \vert^2\varphi$ satisfies the requirement in \eqref{q2}. As $q$ is symmetric $Q_2$ extends to a symmetric bilinear form on $\Lambda^4_{27}$. Part (i) of the claim is thus proved. To prove (ii) we compute 
\begin{equation*}
\begin{split}
Q(\alpha)=&\langle Q_2(\alpha),\star \alpha\rangle=-\langle i(q_0(\alpha,\alpha)), \star \alpha \rangle=-2\langle (q_0(\alpha,\alpha)), i^{-1}(\star \alpha) \rangle\\
=&-2\langle (q(\alpha,\alpha)), i^{-1}(\star \alpha) \rangle
\end{split}
\end{equation*}
by taking into account \eqref{idg21} and that pure trace components can be ignored by type considerations.
\end{proof}

At this stage a few remarks are in order.
\begin{rema} \label{rmks1}
\begin{itemize}
\item[(i)]Proposition \ref{pol} proves directly existence for $b_2$; an a priori proof of this fact  stems from having $b_2$ arising as the second derivative, in a suitable sense, of Hitchin's duality map.
\medskip
\item[(ii)] In the last  section of the paper it will convenient to work with the cubic polynomial 
$P:\Lambda^{3}_{27} \to \mathbb{R}$ given by $P(\beta)=Q(\star \beta)$. By part (ii) in Proposition \ref{pol} and some linear algebra
this is computed from 
\begin{equation*} \label{Pp}
P(\beta)=2\langle p(\beta,\beta),i^{-1}(\beta)\rangle
\end{equation*}
where the symmetric bilinear form $p:\Lambda^3_{27} \times \Lambda^3_{27} \to \Sym^2\T$ is determined from 
the equation $p(\beta,\beta)(v_1,v_2)=\langle v_1 \lrcorner \beta, v_1 \lrcorner \beta \rangle$. 
\medskip
\item[(iii)] Another way of thinking about the bilinear form $b_2$ is based on the following observation. After suitably contracting 
\eqref{q2} and using \eqref{idg21} we obtain 
\begin{equation*}
\begin{split}
\langle b_2(\star \beta_1, \star \beta_2),i(S_3) \rangle=&\langle (S_3)_{\star}\beta_1,\beta_2 \rangle+\langle (S_3)_{\star}\beta_2,\beta_1 \rangle\\
=&\langle (S_3)_{\star}(S_1)_{\star}\varphi, (S_2)_{\star}\varphi\rangle+\langle (S_3)_{\star}(S_2)_{\star}\varphi, (S_1)_{\star}\varphi\rangle
\end{split}
\end{equation*}
with $\star \beta_k=i(S_k) \in \Lambda_{27}^3, k=1,2$ and $S_3 \in \Sym^2_0\T$. Using repeatedly relations of the type 
$(S_3)_{\star}(S_1)_{\star}\varphi=(S_1)_{\star}(S_2)_{\star}\varphi+[S_3,S_1]_{\star}\varphi$ together with 
$\{   F_{\star}\varphi  :  F \in \Lambda^2 \}  \subseteq \Lambda^3_7$ we see this is symmetric in $S_1,S_2,S_3$.
In other words $P$ is induced by a $G_2$-invariant tensor in $\Sym^3(\Lambda^3_{27})$. Since $G_2$ acts without fixed vectors 
on $\Lambda^3_{27}$, this third order symmetric tensor actually belongs to $\Sym^3_0(\Lambda^3_{27})$. Classical invariant could then be used to recover $P$ from octonian multiplication.
\end{itemize}
\end{rema}
 To conclude this section record the following simple observations which will be used later on in the paper.
\begin{lema}\label{injective}
If $  \beta\in \Lambda^3$ satisfies $\beta \w (v \lrcorner \psi) = 0 $ for all  vectors $v \in T$, then $\beta$ has to vanish.
\end{lema}
\begin{proof}
By a suitable contraction $\beta \w A_{\star} \psi=0$ for all $A \in \End \T$. As the action of $\End \T $ on $\psi$ spans $\Lambda^4$ we get 
$\beta \w \Lambda^4=0$ thus $\beta=0$.

\end{proof}
\begin{lema} \label{a1} 
Let $\alpha=\lambda \psi+V \w \varphi+\alpha_{27} \in \Lambda^4$. Then $\alpha \w (v \lrcorner \psi)= - 4 g(V,v)\vol, v\in \T$.
\end{lema}
\begin{proof}
We have $(\alpha_{27}+\lambda \psi) \w (v \lrcorner \psi)=-(\alpha_{27}+\lambda \psi) \w \star (v \w \varphi)=-g(\alpha_{27}+\lambda \psi,v \w \varphi)\vol=0$. 
The claim follows from the algebraic identity $\varphi \w (v \lrcorner \psi)= - 4v \lrcorner \vol$. 
\end{proof}

\medskip

%
%
\subsection{The Lie derivative} \label{lie}
Throughout this paper 
we systematically denote the various algebraic 
components in the codifferential with $\di^{\star}_7\alpha:=(\di^{\star}\alpha)_7 $ etc. for $\alpha \in \Omega^{\star}(M)$, w.r.t. the type decompositions  $\Omega^k(M)=\Omega^k_1 \oplus \Omega^k_7 \oplus \Omega^k_{27}, k=3,4$.  
The first objective is to render explicit the structure of Lie derivatives $\li_X\psi, X \in \Gamma(\T M)$ according to the splitting $\Omega^4(M)=\Omega^4_1 \oplus \Omega^4_7 \oplus \Omega^4_{27}$.
Let  $L:\Gamma(\T M) \to \Gamma(\T M)$ denote the first order differential operator  determined from 
\begin{equation*}
\di_7X=\tfrac{1}{3}L(X) \lrcorner \varphi.
\end{equation*}
A second first order differential operator of relevance is the trace free part of the Lie derivative of the metric 
\begin{equation}\label{Sx}
X \in \Gamma(\T M) \mapsto S_X=\tfrac{1}{2} \li_Xg+\tfrac{1}{7} (\di^{\star}X) \,g \in \Gamma(\Sym^2_0\T M)
\end{equation}
as showed below.
\begin{lema} \label{lie2}
We have 
\begin{equation} \label{lie0}
\li_X\psi=-\tfrac{4}{7}(\di^{\star}X)\, \psi+(\tfrac{1}{2}L(X)-\tfrac{\tau_0}{4}X)\w \varphi+(S_X)_{\star}\psi
\end{equation}
and
\begin{equation} \label{lie3} 
\di^{\star}(X \w \psi)=
\tfrac{3}{7}(\di^{\star}X)\psi+(-\tfrac{1}{2}L(X)-\tfrac{3\tau_0}{4}X)\w \varphi+(S_X)_{\star}\psi
\end{equation}
whenever $X \in \Gamma(\T M)$.

\end{lema}
\begin{proof}
Recall the local expressions $\di=e_i \w \nabla_{e_i}$ and $\di^{\star}=-e_i \lrcorner \nabla_{e_i}$ where $\nabla$ denotes the Levi-Civita connection of $g$. At the same time $\nabla_U\psi=-\tfrac{\tau_0}{4}U \w \varphi$ with $U \in \T M$. Direct computation based on these facts leads to 
\begin{equation*}
\begin{split}
& \di(X \lrcorner \psi)=(\nabla^t X)_{\star}\psi-\tfrac{\tau_0}{4}X \w \varphi\\
& \di^{\star}(X \w \psi)=(\di^{\star}\!X)\psi+(\nabla X)_{\star}\psi-\tfrac{3\tau_0}{4}X \w \varphi
\end{split}
\end{equation*}
where $\nabla^tX$ denotes the transpose of $\nabla X \in \End(\T)$ w.r.t. $g$, i.e. the difference of the symmetric and the anti-symmetric part of $\nabla X$. All claims follow now from the type decomposition 
\begin{equation*}
\nabla X=
\tfrac{1}{2}\di_{14}\! X +\tfrac{1}{2} \di_7\! X -\tfrac{1}{7}(\di^{\star}\!X) \,1_{\T M}+S_X
\end{equation*}
combined with 
the algebraic identity $(v \lrcorner \varphi)_{\star}\psi=-3 v \w \varphi, v \in \T$.
\end{proof}
Recall that the divergence operator $\delta :\Gamma(\Sym^2 \T M) \to \Gamma(\T M)$ is defined according to the convention $\delta S:= - (\nabla_{e_i}S)e_i$. Below we work out how $\di^{\star}$ and $\delta$ relate via the algebraic isomorphism 
$\Sym^2_0T \to \Lambda^4_{27}$. In the rest of this paper we indicate with 
$(\cdot, \cdot)$ the $L^2$-inner product on tensor fields.
\begin{lema}\label{deltaS}
For  any symmetric tensor in $S \in \Gamma(\Sym^2_0 \T M)$ we have
\begin{equation} \label{ds7}
\di_{7}^{\star}(S_{\star}\psi)=\tfrac{1}{2}(\delta S) \lrcorner \psi.
\end{equation}
\end{lema}
\begin{proof}
Pick $X \in \Gamma(\T M)$ and record that $\li_Xg = \delta^*X$ where $\delta^*$ is the formal adjoint of $\delta$. Writing $\di_{7}^{\star}(S_{\star}\psi) = Z \lrcorner \psi$ ensures that $(\di_{7}^{\star}(S_{\star}\psi), X\lrcorner \psi) = 4(Z, X)$. 
At the same time  
\bea
(\di_{7}^{\star}(S_{\star}\psi) , X \lrcorner \psi ) &=& (S_{\star} \psi, \di(X \lrcorner \psi )) 
= - (S_{\star} \psi, \di \star (X \w \varphi)  )
= - (\star S_{\star} \psi, \di^{\star}  (X \w \varphi)  )\\[1ex]
&=&
(S_\star \varphi, (S_X)_\star \varphi) = 2(S, S_X) =  2(S, \delta^* X ) = 2 (\delta S, X)
\eea
where we used \eqref{idg21} and that $S$ is trace free. Hence $Z = \tfrac12 \delta S$.
\end{proof}

%
%
%
%
\section{Deformation theory}
\subsection{Curves of $G_2$-structures} \label{curves}
Consider a compact manifold $M^7$. Assume that it is equipped with a $G_2$-structure $(\varphi,\psi) \in \Omega^3(M) \oplus \Omega^4(M)$ such that 
\begin{equation} \label{str1}
\di\!\varphi=\tau_0   \psi.
\end{equation}
Assume that $(\varphi_t,\psi_t) \in \Omega^3(M) \times \Omega^4(M)$ is a small time deformation of $(\varphi,\psi)$ satisfying \eqref{str1} and having constant volume $\vol \in \Omega^7(M)$. This can be assumed w.l.o.g by Moser's theorem. Consider the truncated Taylor series 
\begin{equation*}
\psi_t=\psi+t\psi_1+\tfrac{t^2}{2}\psi_2 + O(t^3).
\end{equation*}

From here we obtain the truncated Taylor series for $\varphi_t$ as follows. First differentiate the algebraic identity  
\begin{equation} \label{alg1}
\varphi_t \w (X \lrcorner \psi_t)=-4 \, X \lrcorner \vol
\end{equation}
for some vector field $X$. At $t=0$ we obtain 
$$
0 = \dot \varphi \w (X \lrcorner \psi) + \varphi \w (X \lrcorner \psi_1) =  \dot \varphi \w (X \lrcorner \psi) - \widehat{\psi_1} \w (X \lrcorner \psi)
$$
by taking into account $\dot \psi = \psi_1$ and \eqref{hat2}. Since the wedge product with $X \lrcorner \psi $ is injective in the sense of Lemma \ref{injective}  we find 
 $\dot \varphi = \widehat{\psi_1}$. Differentiating at second order in \eqref{alg1} yields at $t=0$
\bea
0 &=& \ddot \varphi \w (X \lrcorner \psi) + 2 \, \dot \varphi \w (X \lrcorner  \dot \psi) + \varphi \w (X \lrcorner \ddot \psi)\\[1ex]
&=& \ddot \varphi \w (X \lrcorner \psi) + 2 \, \widehat{\psi_1} \w (X \lrcorner  \psi_1) + \varphi \w (X \lrcorner  \psi_2)\\[1ex]
&=&
(\ddot \varphi   -Q_2(\psi_1)  - \widehat{\psi_2} )\w (X \lrcorner \psi)
\eea
after succesive use of $\ddot \psi = \psi_2$ combined with \eqref{hat2} and the definition of the quadratic form $Q_2$ from \eqref{q2}.
Thus we find that $\ddot \varphi = Q_2(\psi_1) + \widehat{\psi_2}$ again by Lemma \ref{injective}. Summarising we obtain the following well-known parametrisation for  $\varphi_t$ essentially contained in \cite[Propn.5]{Br}.
\begin{lema}\label{param} 
The truncated Taylor series for $\varphi_t$ reads
\begin{equation} \label{param2}
\varphi_t=\varphi+t\widehat{\psi_1}+\tfrac{t^2}{2}(\widehat{\psi_2}+Q_2(\psi_1)) + O(t^3).
\end{equation}
\end{lema}

\bigskip

\noindent
This makes it straightforward to determine the differential operator which governs the deformation theory of nearly $G_2$ structures by  differentiating the structure equations. Let $D:\Omega^4(M) \to \Omega^4(M)$ 
be defined by
\begin{equation*}
D\alpha:=\di\! \widehat{\alpha}-\tau_0 \alpha.
\end{equation*}
In fact, differentiating at order two in $\di\!\varphi_t=\tau_0 \psi_t$ whilst using \eqref{param2} yields at $t=0$ 
\begin{equation} \label{sys}
D\psi_1=0\\
\quad \mbox{and}\quad
D\psi_2= - \di\!  Q_2(\psi_1).
\end{equation}
This second equation prompts out the following 
\begin{defi} \label{un2}
An element $\psi_1 \in \Ker(D)$ is unobstructed to second order provided there exists $\psi_2 \in \Omega^4(M)$ solving 
$D\psi_2=-\di Q_2(\psi_1)$.
\end{defi}
A first key step in the study of such objects is to observe that the first order differential operator $D$ is sub-elliptic  in the sense of \cite{Besse}. Indeed, its principal symbol which is given by $\xi \in \Lambda^1M \mapsto 
\sigma(\xi)\alpha=\xi \wedge \widehat{\alpha}$ is injective. Thus, Hodge theory for sub-elliptic operators makes that 
$$ \Omega^4(M) = \ker(D^{\star}) \oplus \im(D)$$ 
orthogonally w.r.t. the $L^2$-inner product. In particular the equation $D\psi_2= -\di Q_2(\psi_1)$ can be solved for $\psi_2$ if and only if $\di Q_2(\psi_1) \perp \ker(D^{\star})$, that 
is 
\begin{equation}\label{cond}
Q_2(\psi_1) \perp \di^{\star}\ker(D^{\star})
\end{equation}
w.r.t. the $L^2$-inner product. Rendering this constraint more transparent relies on the explicit computation of  
 $\di^{\star}\ker(D^{\star})$ which is the main objective in the next section.
%
\subsection{Structure of the operators $D,D^{\star}$} \label{dec-slice}
%
The following general observation will be repeatedly used in this section. It stems from a first order version of Lemma \ref{param2} where the volume is allowed to vary with the deformation.
\begin{lema} \label{st1}
We have $D(\li_X \psi)=\di(\di^{\star}\! X \varphi)$ for $X \in \Gamma(\T M)$.
\end{lema}
\begin{proof}
Consider the flow $(\Phi_t)_{t \in \mathbb{R}}$ of $X$ together with the nearly $G_2$-structure defined by $\varphi_t=\Phi_t^{\star}\varphi$ \, and $ \psi_t=\Phi_t^{\star} \psi$, with volume form $\vol_t=\Phi_t^{\star}\vol$. Differentiating at $t=0$ in the algebraic identity 
$\varphi_t \w (Y \lrcorner \psi_t)= -4  \, Y \lrcorner \vol_t$ yields
$$
\dot \varphi \wedge (Y \lrcorner \psi) + \varphi \wedge (Y\lrcorner \li_X\psi) = 4  (\di^*\!X)\, Y \lrcorner \vol = -  (\di^*\!X)\, \varphi \w (Y \lrcorner \psi)
$$
where we used $\dot{\vol_t}=\li_X\!\vol=-(\di^{\star}\!X) \vol$, \,$\dot{\psi_t}=\li_X\!\psi$ at $t=0$ and once again the algebraic identity.
Then Lemma \ref{hat} leads to
$
[\dot \varphi  - \widehat{\li_X\psi } + (\di^*\!X) \,\varphi] \w  (Y\lrcorner \psi) = 0
$.
By Lemma \ref{injective} we obtain $ \widehat{\li_X\psi } = \dot \varphi + (\di^*\!X) \,\varphi$.
This gives
$$
D(\li_X \psi) = \di  \widehat{\li_X\psi } - \tau_0  \li_X\psi = \di \dot \varphi +\di(\di^*\!X\,\varphi) - \tau_0  \dot \psi 
$$
and the claim follows by differentiating in $\di\!\varphi_t=\tau_0 \psi_t$. 
\end{proof}
Let $\Kk:=\mathfrak{aut}(M,g)$  be the space of Killing vector fields and denote with $\Kk^{\perp}$ its $L^2$-orthogonal within $\Gamma(\T M)$. We assume that the structure is {\it{proper}} in the sense of \cite{FKMS}, i.e. $(M, g)$ admits a one-dimensional space of Killing spinors and in particular we have
\begin{equation} \label{gen}
\Kk \subseteq \aut(M,\varphi).
\end{equation}
Thus the Lie derivative of $\varphi$ and $\psi$ in the direction of Killing vector fields vanishes. This condition leads to an important simplification in subsequent 
calculations. A few more background facts we need are related to the $L^2$-decomposition 
$$\Gamma(\Sym^2\T M) = \im \delta^* \oplus \ker \delta.$$
This plays a key role in proving the Ebin slice theorem. In presence of an Einstein metric with positive scalar curvature it can be refined as follows, mainly due to Obata's theorem.
\begin{pro} \label{inv-Q} 
Assume that the compact Einstein manifold $(M^7,g)$ with $\scal>0$ is not isometric to the standard sphere. Then  $\mathcal{K}^{\perp}$ embeds in $\Gamma(\Sym^2_0M)$ via $X \mapsto S_X$,
with $S_X$ defined in \eqref{Sx}, and 
\begin{equation} \label{decS}
\Gamma(\Sym^2_0M)=\{S \in \Gamma(\Sym^2_0 \T M) : \delta S \in \mathcal{K}\} \oplus \mathcal{K}^{\perp}
\end{equation}
orthogonally w.r.t. $L^2$-inner product.
\end{pro}
\begin{proof}
This can be extracted from \cite[Thm.4.60]{Besse} and related material in that reference. In fact the statement is true 
in arbitrary dimension. We outline some of the arguments required mainly for the convenience of the reader.

Computation using Bochner's formula on $1$-forms yields 
$$ 
\delta(\li_X\!g)=(\Delta-2\Ric)X+\di\di^{\star}\!X
$$
with $X \in \Gamma(\T M)$. Thus 
\begin{equation} \label{delSX}
\delta S_X = \delta(\tfrac12 \li_Xg +\tfrac17 \di^*\!X\,g) = \tfrac12 \Delta X - \mathrm{Ric} X  + \tfrac12 \di\!\di^*\!X- \tfrac17 \di\! \di^*\!X
=:AX
\end{equation}
where the operator $A:\Gamma(\T M) \to \Gamma(\T M)$ is given by $A=\tfrac12 \Delta  - \mathrm{Ric}   + \tfrac{5}{14} \di\!\di^*$. 
Because $g$ is an Einstein metric $A$ preserves the splitting $\Omega^1(M)=\ker(\di^{\star}) \oplus \im \di$. On $\ker(\di^{\star})$ it is clear that $\ker(A)=\Kk$. At the same time 
we have $A \circ \di=\frac{6}{7}\di(\Delta-\frac{\scal}{6})$ on the space $\Cc^{\infty}(M)$. As $\scal>0$ and $g$ does not have constant sectional curvature Obata's Theorem ensures that the operator 
$\Delta-\frac{\scal}{6}$ is invertible on $\Cc^{\infty}(M)$ thus $A$ is invertible on $\mathcal{K}^{\perp}$. In particular the map $X \in \mathcal{K}^{\perp} \mapsto S_X$ injective. To conclude, pick $S \in \Gamma(\Sym^2_0 \T M)$ and decompose  the vector field $\delta S$ as  $\delta S=K+Y$ with $K \in \Kk$ and $Y \in \Kk^{\perp}$ and
choose $Z \in \Kk^{\perp}$ such that $AZ=Y$. By \eqref{delSX} we have $\delta(S-S_Z)=K+Y-AZ=K \in \mathcal{K}$ and the claim is proved. 
\end{proof}

It turns out that the splitting \eqref{decS} replicates at the level of $4$-forms, in a way which is consistent with both the algebraic isomorphism $\Sym_0^2\T  \to \Lambda_{27}^4$ and the Hodge decomposition for $D$.
To describe how this works consider the  spaces 
\begin{equation*} 
\E:=\{(\li_X\psi)_{27} : X \in \Kk^{\perp}\}
\qquad \mbox{and} \qquad
 \F:=\{\alpha \in \Omega^4_{27} : \di^{\star}_7 \alpha \in \Kk \lrcorner \psi \} .
\end{equation*}
For ease of reference we write $\Kk \lrcorner \psi = \{ X \lrcorner \psi : X \in \Kk\}$ and
$\Omega^4_{7^{\prime}}:=\Kk^{\perp} \w \varphi$, as well as $\Omega^{4}_{1 \oplus 7^{\prime}}=\Omega^4_1 \oplus \Omega^4_{7^{\prime}}$. 
\begin{lema} \label{dec1}
We have an $L^2$-orthogonal splitting 
\begin{equation} \label{main-s}
\Omega^4(M) = (\Kk \w \varphi) \oplus (\Omega^4_{1 \oplus 7^{\prime}} \oplus \E) \oplus \F.
\end{equation}
\end{lema}
\begin{proof}
It is enough to check that $\Omega^4_{27}=\E \oplus \F$. Clearly $\E$ and $\F$ are $L^2$-orthogonal. Indeed for $\alpha\in \F$ and
$X\in \Kk^\perp$, i.e.  $(\li_X \psi)_{27} \in \E$ we compute
$$
(\alpha, (\li_X \psi)_{27} ) = (\alpha, \li_X \psi) = (\alpha, \di (X \lrcorner \psi)) = (\di^{\star}_7\!\psi, X \lrcorner \psi) = 0
$$
according to the definition of $\F$.

\medskip

That $\E \oplus \F$ spans 
$\Omega^4_{27}$ is a consequence of Proposition \ref{inv-Q} as outlined below.   
Let $\alpha=S_\ast \psi$ belong to $\Omega^4_{27}$. 
Decomposing the tensor field $S=\widetilde{S}+S_X$ according to \eqref{decS} we find 
\begin{equation*}
\alpha=\widetilde{S}_{\star}\psi+(S_X)_{\star}\psi=\widetilde{S}_{\star}\psi+(\li_X\psi)_{27}
\end{equation*}
after also taking into account \eqref{lie0}. The last summand belongs, by definition, to $\E$. To conclude we use \eqref{ds7} to check that  
$
\di_{7}^{\star}(\widetilde{S}_\ast \psi) = \tfrac12 \delta \widetilde{S}  \lrcorner \psi \in \Kk \lrcorner \psi.$ 
Thus $\widetilde{S}_\ast \psi \in \F$ and the claim is proved.
\end{proof}

The next objective is to determine how the decomposition above is acted on by the operators $D$ and $D^{\star}$. This is based on the following 
\begin{pro} \label{inv-D}
We have 
\begin{itemize}
\item[(i)] the splitting \eqref{main-s} is preserved by the operators $D$ and $D^{\star}$
\item[(ii)] $D$ is self-adjoint on $\F$, that is $D^{\star}_{\vert \F}=D_{\vert \F}$.
\end{itemize}
\end{pro}
\begin{proof}
(i) first we check that $D$ preserves  $\Omega^4_1\oplus (\Kk^{\perp} \w \varphi) \oplus \E$, the second summand of the
splitting \eqref{main-s}. By direct computation we obtain
\begin{equation} \label{block1}
D(f\psi)=-\di\!f \w \varphi- 2\tau_0 f\psi\\
\qquad \mbox{and} \qquad
D(X \w \varphi)=- \li_X \psi- \tau_0X \w \varphi
\end{equation}
with $(f,X) \in \Cc^{\infty}(M) \times \Gamma(\T M)$.  Note that $\di\!f \in \Kk^\perp$ since Killing vector fields are co-closed. Clearly $\li_X \psi \in \Omega^4_{1 \oplus 7} \oplus \E$. Moreover we see
$$ 
( \li_X\psi, K \w \varphi )= (X \lrcorner \psi, \di^{\star}(K \w \varphi))=-(X \lrcorner \psi,\star \li_K \psi)=0
$$
for all $K \in \Kk$ by using \eqref{gen}. Thus 
\begin{equation} \label{inc0}
\li_X \psi \in \Omega^4_{1 \oplus 7^{\prime}} \oplus \E.
\end{equation}
Combining this observation with \eqref{block1} leads to 
\begin{equation} \label{inc1}
D(\Omega^4_{1 \oplus 7^{\prime}}) \subseteq \Omega^4_{1 \oplus 7^{\prime}} \oplus \E.
\end{equation} 
It remains to consider the action of $D$ on $\E$. Here we find that
$$
D((\li _Y \psi)_{27})=D(\li_Y \psi)-D((\li_X \psi)_{1 \oplus 7^{\prime}}) \in \Omega^4_{1 \oplus 7^{\prime}} \oplus \E \ ,
$$ 
by using Lemma \ref{st1} and \eqref{inc1} above. 

\medskip

The first summand in    \eqref{main-s} is preserved by $D$ because of the  second equation in
\eqref{block1} and $\li_X \psi = 0$ for  Killing vector fields $X$.

\medskip

Finally we have to show that the operator $D$ also preserves $\F$, i.e. the third summand of  the decomposition  \eqref{main-s}. Here we take $(\alpha,X) \in \F \times \Kk^{\perp}$ and compute the $L^2$-product
\begin{equation*}
( D\alpha, (\li_X \psi)_{27})= -(\di \star \alpha, (\li_X \psi)_{27})= -(\alpha, \star \di^{\star}
((\li_X \psi)_{27})) =(\alpha, D((\li_X \psi)_{27})).
\end{equation*} 
where we have taken into account that $\alpha \perp (\li_X \psi)_{27}$, which is true since $\F \perp \E$ and  $\alpha\in \F$, $ (\li_X \psi)_{27}\in \E$ by assumption. 
As showed above 
$D \E \subseteq \Omega^4_{1 \oplus 7'} \oplus \E $ which  is orthogonal to $\F$. Thus the equation above shows $D\F \perp \E$, i.e. we already have 
$D\F \subseteq \Omega^4_{1 \oplus 7} \oplus \F$. Next we will show that $D\F$ is orthogonal to $\Omega^4_{7'}$, i.e. to forms $X \w \varphi$ with $X \in \Kk^\perp$.
For a  $\alpha \in \F \subset \Omega^4_{27}$ we have $D\alpha = - \di\! \star \alpha - \tau_0 \alpha $. Hence, with   $X\in \Kk^\perp$ we obtain
$$
(D\alpha, X \w \varphi) = -(\di\! \star  \alpha - \tau_0 \alpha, X \w \varphi) = -(\star \di\! \star \alpha, \star(X \w \varphi)) = (\di^*_7\! \alpha, X \lrcorner \psi ) = 0
$$
by the defining condition of $\F$. But $D\alpha$ is also orthogonal to any $4$-form $X \w \varphi$ for $X\in \Kk$. Indeed
if $X \in \Kk$ we have $\li_X\psi=0$, since the structure is proper, and we obtain for any $\alpha\in \Omega^4_{27}$ that
$$
(D\alpha, X \w \varphi) = - (\di\! \star \alpha,  X \w \varphi ) =  (\di^*\! \alpha, X \lrcorner \psi) =  (\alpha, \di(X \lrcorner \psi)) = (\alpha, \li_X \psi) = 0 \ .
$$
There remains to prove that $D\F$ is orthogonal to $\Omega^4_1$. Letting  $f$ be some function on $M$ and 
$\alpha\in \F$ we  similarly compute
$$
(D\alpha, f\psi) = -(\di \star \alpha, f\psi ) = - (\di^*\! \alpha, f\varphi) = -(\alpha, \di(f\varphi)) = - (\alpha, \di\! f \w \varphi + f\tau_0 \psi) = 0 \ .
$$
 Thus $D\F \subseteq \F$ as claimed.  Since the splitting \eqref{main-s} is $L^2$-orthogonal and preserved by $D$ it must also be preserved by $D^{\star}$.
\medskip 

(ii) follows from $\star \di^{\star}=\di \star$ on $\Omega^4(M)$ and $D\F \subseteq \F$.
\end{proof}

\medskip
The infinitesimal deformation space of the nearly $G_2$ structure is defined according to 
\begin{equation*}
\F_4:=\ker(D) \cap \Omega^4_{27}.
\end{equation*}
Rewriting
\begin{equation*}
\F_4=\{\alpha \in \Omega^4_{27} : \star \,\di\, \star \, \alpha = - \tau_0 \star \alpha \}=\{
\alpha \in \Omega^4_{27} : \di^*\!\alpha = - \tau_0 \star \alpha \}
\end{equation*}
has several consequences. Firstly $\di_7^{\star}$ vanishes automatically on $\F_4$, thus $\F_4 \subseteq \F$; secondly 
$$\F_3 := \star \F_4=\{\beta \in \Omega_{27}^3 : \star \di \beta=-\tau_0 \beta \}$$ is exactly the space of infinitesimal deformations considered in \cite{AlS}. Lastly, $\F_4$ is a subspace
of the eigenspace of the Laplace operator acting on $4$-forms for the eigenvalue $\tau^2_0$. In particular $\F_4$ is finite dimensional.

\medskip

An important first consequence of Proposition \ref{inv-D} is the following 
\begin{coro}We have 
\begin{equation*} 
\ker(D)=\ker(D_{\vert \Omega^4_{1 \oplus 7^{\prime}} \oplus \E}) \oplus \mathcal{F}_4
\end{equation*}
as well as 
\begin{equation*} 
\ker(D^{\star})=\ker(D^{\star}_{\vert \Omega^4_{1 \oplus 7^{\prime}} \oplus \E}) \oplus \mathcal{F}_4.
\end{equation*}
\end{coro}
\begin{proof}
We have already seen that $D=- \tau_0\,\id$ on $\Kk \wedge \varphi$. Thus $D$ has no kernel on the first summand of  \eqref{main-s} and the statement
follows since $D$ preserves
the decomposition  \eqref{main-s}. The second part of the claim follows from having the restriction of $D$ to $\F$ self-adjoint. 
\end{proof}

It is now straightforward to determine the action of $D$ on $\Omega^4_{1 \oplus 7^{\prime}} \oplus \E$ as follows. The main observation here is that 
\begin{pro} We have an identification map 
\begin{equation} \label{id}
(f,X,Y) \in \Cc^{\infty}(M) \oplus \K^{\perp} \oplus \Kk^{\perp} \mapsto f\psi+X \wedge \varphi+\li_Y\! \psi \in \Omega^4_{1 \oplus 7^{\prime}} \oplus \E
\end{equation}
w.r.t. which 
\begin{equation} \label{Ds1}
D(f,X,Y)=(\tau_0 \di^{\star}Y- 2\tau_0 f,\di(\di^{\star}Y-f)-\tau_0X,-X).
\end{equation}
\end{pro}
\begin{proof}
For ease of reference indicate with $\iota$ the map in \eqref{id}. Clearly $\iota(f,X,Y)=0$ forces $(\li_Y \psi)_{27}=0$ thus $S_Y=0$ by \eqref{lie0}. By Proposition \ref{inv-Q} it follows that $Y=0$ which ensures that $f$ and $X$ vanish as well.
That $\iota$ is surjective follows, via the definition of $\E$,  from $(\li_X \! \psi)_{27}=\li_X\! \psi-(\li_X \! \psi)_{1\oplus 7^{\prime}}$ whenever $X \in \Kk^{\perp}$, which allows absorbing the $27$-component of $\li_X\!\psi$ into $\E$. Note that we also use \eqref{inc0} to see that $\li_X\! \psi$ has no component on $\K \w \varphi$.
The claim in \eqref{Ds1} is now granted by Lemma \ref{st1} and \eqref{block1}.
\end{proof}
An easy argument based on \eqref{Ds1} shows that 
\begin{equation} \label{kD}
\ker(D)=\{\li_X \psi : X \in \Kk^{\perp}, \di^{\star}X=0 \} \oplus \mathcal{F}_4.
\end{equation}
In particular,  this decomposition shows that infinitesimal deformations can be  normalised,  up to the action of volume preserving  diffeomorphisms,  to lie in $\F_4$.
This is consistent with the  approach in \cite{AlS}  where the first variation of nearly $G_2$ metrics has been normalised to lie in the Ebin slice for Einstein metrics, thus leading to the identification of the space of infinitesimal deformations with $\F_3=\star \F_4$.

\medskip

%
\subsection{Computation of $\ker(D^{\star})$} \label{kerds}
%
Computing the kernel of $D^{\star}$ on $\Omega^4_{1 \oplus 7} \oplus \E$ requires a bit more work as the map in \eqref{id} is {\it{not}} an isometry w.r.t. the canonical $L^2$-inner product on $\Cc^{\infty}(M) \times \Kk^{\perp} \times \Kk^{\perp}$.
The restriction of $D^{\star}$ to $\Omega^4_{1 \oplus 7^{\prime}} \oplus \E$ can be understood similarly to the restriction of $D$, using a slightly different parametrisation for the latter space. 
\begin{lema} \label{lLl}
The spaces spanned by $\di(\K \w \varphi)$ and $\K^{\perp} \w \psi$ are $L^2$-orthogonal. 
\end{lema}
\begin{proof}
Pick $K \in \K$; since $\li_K\psi = 0$  we have $\di_7K=\frac{\tau_0}{6}K \lrcorner \varphi$ as guaranteed by \eqref{lie0}. Thus
\bea
\di (K \w \varphi)&=& \di\!K \w \varphi \,-\, \tau_0 K \w \psi \;=\; (\di_{14}K) \w \varphi \,+\, (\di_7K) \w \varphi-\tau_0 K \w \psi \\[.5ex]
&=&
\star \di_{14}K \,+\, \tfrac{\tau_0}{6}(K \lrcorner \varphi) \w \varphi  \,-\, \tau_0 K \w \psi
\eea 
by taking into account that $\Lambda_{14}^2=\{\alpha \in \Lambda^2 : \alpha \w \varphi=\star \alpha\}$. 
Hence the scalar product  reads
\bea
\langle Y \w \psi,\di(K \w \varphi)\rangle &=& \langle \star (Y \i \varphi), \di(K \w \varphi)\rangle \;=\; \tfrac{\tau_0}{6}
\langle \star (Y \i \varphi), (K \lrcorner \varphi) \w \varphi\rangle  \,-\,  \tau_0 \langle Y \lrcorner \varphi, K \lrcorner \varphi \rangle\\[.5ex]
&=&  
 -\tfrac{4\tau_0}{3} \langle Y \i \varphi, K \i  \varphi \rangle  \:=\;
-4\tau_0 g(K,Y)
\eea
with $Y \in \Gamma(\T M)$ and the claim follows by integration, whilst taking $Y \in \K^{\perp}$.
\end{proof}
\begin{pro} \label{kerDs4}
We have
\begin{equation*}
\ker(D^{\star})=\{     \tau_0 Y \w \varphi  +  \di^{\star}(Y \wedge \psi) : Y \in \Kk^{\perp}, \di^{\star}Y=0\} \oplus \F_4.
\end{equation*}
\end{pro}
\begin{proof}
Pick $\alpha \in \Omega^4_{1 \oplus 7^{\prime}} \oplus \E$. Since we are ultimately interested in the space spanned by $\di^{\star}\!\alpha$ 
it is convenient to parametrise 
$\alpha=f\psi+X \w \varphi+\di^{\star}(Y \w \psi)$ with $(f,X,Y) \in \Cc^{\infty}(M) \times \Kk^{\perp} \times \Kk^{\perp}$. The existence proof relies on \eqref{lie3} which ensures that $\di_{27}^{\star}(Y \w \psi)=(\li_Y\! \psi)_{27}$ and on having $\di^{\star}(Y \w \psi)$ 
orthogonal to $\K \w \varphi$, as granted by Lemma \ref{lLl}.
Uniqueness is entirely similar to the argument used to establish \ref{id}. Because $\di \circ D=-\tau_0 \di$ we have 
$$D^{\star}\circ \di^{\star}=- \tau_0 \di^{\star} \ \mbox{on} \ \Omega^5(M).
$$ 
On the other hand, using \eqref{hat3} of Lemma \ref{hat} and an easy $L^2$-orthogonality argument we compute the three components
of $D^*(X \wedge \varphi)$  in $\Omega^4(M)$  as
\begin{equation*} 
D^{\star}_{1}(X \w \varphi)=(\li_X \psi)_1, \;  \; D^{\star}_{7}(X \w \varphi)=-(\li_X\psi)_{7}-\tau_0 \,X \w \varphi, \; \;D^{\star}_{27}(X \w \varphi)=(\li_X\psi)_{27}.
\end{equation*}
This leads to 
\begin{equation*}
D^{\star}(X \w \varphi)=\di^{\star}(X \w \psi)+\tfrac17(\di^{\star}\!X)\psi
\end{equation*}
after comparing the type components in $\li_X\!\psi$ and $\di^{\star}(X \w \psi)$ according to \eqref{lie0} and \eqref{lie3}.
On the other hand
\begin{equation*}
D^{\star}(f\psi)=\di\!f \w \varphi-  2\tau_0 f\psi
\end{equation*} 
by direct computation.
Thus having $D^{\star}\alpha=0$ for a  form $\alpha \in \Omega^4_{1 \oplus 7^{\prime}} \oplus \E$ reads, in terms of the  parametrisation above,
$$ 
\di^{\star}\!((X-\tau_0 Y) \w \psi)+\di\! f\w \varphi+(\tfrac17\di^{\star}X  - 2\tau_0 f) \psi=0.
$$
Projecting the last equation onto $\Omega^4_{27}$ shows, by using \eqref{lie3}, that $(S_{X-\tau_0 Y})_\ast \psi=0$, i.e. $S_{X-\tau_0Y}=0$. 
Since  $X, Y \in \Kk^\perp$ by assumption, it follows that $X=\tau_0Y$ by Proposition \ref{inv-Q}. Clearly this entails $f=0$ and $\di^{\star}X=0$ and the claim is proved.
\end{proof}
In particular,  we obtain 
\begin{coro} \label{kdlast}
\begin{equation} \label{kD11}
\di^{\star}\ker(D^{\star})=\{\di^{\star}(Y \wedge \varphi) : Y \in \Kk^{\perp}, \di^{\star}Y=0\} \oplus \F_3.
\end{equation}
\end{coro}
\begin{proof}
Because of Proposition \ref{kerDs4} we only need to check that  $\di^{\star}\mathcal{F}_4=\F_3$ which follows from 
having $\di^{\star}\alpha = -  \tau_0 \star \alpha$ whenever $\alpha$ belongs to $\F_4$. 
\end{proof}
\begin{rema}
For the non-proper nearly $G_2$ structures supported by $3$-Sasaki and Sasaki Einstein metrics
parts of the infinitesimal deformation space 
$\mathcal{F}_4$ have been explicitly computed in cohomological terms by C.van Coevering \cite{coev2}.
\end{rema}

%
%
\section{Proof of Theorem \ref{main}} \label{def2}
%
%

To capture explicitly the properties of the subset of $\ker(D)$ consisting of second order unobstructed  deformations, 
we consider again the Kuranishi-type map 
$\mathbb{K}:\F_4 \to \F^*_4$ as  introduced in \eqref{K1}.
Recall that the map $\mathbb{K}$ depends quadratically on its first argument. We see that a $4$-form $\alpha$ belongs to the zero locus $\mathbb{K}^{-1}(0)$ if and only if
the corresponding $3$-form $Q_2(\alpha)$ is orthogonal to the space $\F_3$. 

\medskip

We first establish the following preliminary 
\begin{lema} \label{dQ2}
Whenever $\alpha \in \F_4$ we have 
\begin{equation*}
\di_7(Q_2(\alpha))=\tfrac{1}{4}\di \vert \alpha \vert^2 \w \varphi.
\end{equation*}
\end{lema}
\begin{proof}
 $X \in \Gamma(\T M)$. Differentiate the defining equation
\begin{equation} \label{Q2} 
Q_2(\alpha) \w (X \lrcorner \psi) \,+\,
2\widehat{\alpha} \w (X \lrcorner \alpha)=0
\end{equation}
 in direction of $e_i$ to obtain 
\begin{equation*}
\nabla_{e_i}(Q_2(\alpha)) \w (X \lrcorner \psi) \,+\, Q_2(\alpha) \w (X \lrcorner \nabla_{e_i}\psi) \,+\,
2\nabla_{e_i}\widehat{\alpha} \w (X \lrcorner \alpha) \,+\, 2\widehat{\alpha} \w (X \lrcorner \nabla_{e_i}\alpha)
\,=\,0.
\end{equation*} 
Note  that $e_i \w (X \lrcorner \nabla_{e_i}\alpha)=\nabla_X\alpha-X \lrcorner \di \alpha = \nabla_X\alpha$ since $\di \widehat{\alpha}=\tau_0 \alpha$ forces $\di \alpha=0$. Similarly we have 
$e_i \w (X \lrcorner \nabla_{e_i}\psi)=\nabla_X\psi$ since $\psi$ is closed. Taking the exterior product with $e_i$ in the displayed equation 
above whilst taking into account that $\nabla_X \psi=-\frac{\tau_0}{4}X \w \varphi$ we arrive at 
$$ 
\di(Q_2(\alpha)) \w (X \lrcorner \psi) \,+\, \tfrac{\tau_0}{4}Q_2(\alpha) \w X \w \varphi
\,+\,
2\di \widehat{\alpha} \w (X \lrcorner \alpha) \,-\, 2\widehat{\alpha} \w \nabla_X\alpha\,=\,0.
$$
As $D\alpha=0$ we have $\di \widehat{\alpha} \w (X \lrcorner \alpha)=\tau_0 \alpha \w (X \lrcorner \alpha)=0$ since $\alpha \w \alpha \in \Omega^8(M)=0$ and $\alpha$ has even degree. Having $\alpha \in \Omega^4_{27}$ ensures that 
$\widehat{\alpha}=-\star \alpha$ thus 
$$\widehat{\alpha} \w \nabla_X\alpha=-g(\alpha, \nabla_X\alpha)\vol=-\tfrac{1}{2}\di \vert \alpha \vert^2(X)\vol.$$ Summarising we obtain
\begin{equation} \label{prell}
\di(Q_2(\alpha)) \w (X \lrcorner \psi) \,+\, \tfrac{\tau_0}{4}Q_2(\alpha) \w X \w \varphi \,+\, (\di \vert \alpha \vert^2)(X)\vol\,=\,0.
\end{equation}
The second (algebraic) summand vanishes since $Q_2(\alpha)$ is orthogonal to $\Lambda^3_7$ due to Proposition \ref{pol},(i).
The claim finally follows from \eqref{prell} and Lemma \ref{a1}.
\end{proof}

The full description of normalised infinitesimal deformations which are unobstructed to second order is contained below. 
\begin{teo} \label{main12}
An element $\psi_1 \in \F_4$ is unobstructed to second order if and only if
\begin{equation}
\psi_1 \in \mathbb{K}^{-1}(0).
\end{equation}
\end{teo}
\begin{proof}
As we already have remarked in \eqref{cond} an infinitesimal deformation described by  $\psi_1$ is unobstructed  to second order if and only if
$Q_2(\psi_1) \perp \di^{\star}\ker(D^{\star})$.
The computation of the latter space in \eqref{kD11} leads to 
$Q_2(\psi_1) \perp \{\di^{\star}(X \w \varphi): X \in \Kk^{\perp}, \di^{\star}X=0\} \oplus \F_3$, i.e.  to
\begin{equation*}
Q_2(\psi_1) \perp  \{\di^{\star}(X \w \varphi): X \in \Kk^{\perp}, \di^{\star}X=0\}
\quad \mbox{and} \quad
Q_2(\psi_1) \perp  \F_3 .
\end{equation*}
The second requirement is by the definition of $\mathbb{K}$ equivalent to $\psi_1 \in \mathbb{K}^{-1}(0)$, whereas the 
first is trivially satisfied by Lemma \ref{dQ2}. Indeed, since $\di^*X=0$ we have
$$
(Q_2(\psi_1), \di^*(X\wedge \varphi)) = (\di_7(Q_2(\psi_1)), X\wedge \varphi) = \tfrac14 (\di|\alpha|^2 \w \varphi, X\w \varphi)= (\di|\alpha|^2, X) = 0  .
$$ 
\end{proof}
This proves Theorem \ref{main1} in the Introduction.
%
%
\section{The Aloff-Wallach space $N(1,1)$} \label{A-W}
%
%
The Aloff-Wallach space $N(1,1) $ can be  described as the normal homogeneous space 
$
G/H =  (\SU(3) \times \SU(2))/(\U(1)\times \SU(2))
$,
where the metric is a multiple of the Killing form of $ \,  \g = \su(3) \oplus \su(2)$. The Lie algebra $\g$ can be written as $\g = \h \oplus \m$. Here
$\h$ is the Lie algebra of $H$ and $\m = \h^\perp$ is its orthogonal complement which can be identified with the tangent space at the origin.
The space $\m$ splits further as $\m = \m_3 \oplus \m_4$ with 
$
\m_3= \spann \{e_1, e_2, e_3\}
$
and
$
\m_4 =  \spann \{e_4, e_5, e_6, e_7\}
$
where $e_1, \ldots, e_7$ is an  orthonormal basis of $\m$ explicitly given in matrix form in \cite{AlS}. Note that $\m_3 = \su(2)_o$ in the notation of that paper. With respect to this basis the nearly $G_2$ structure
on $N(1,1)$ is induced by $\varphi \in \Lambda^3\m$ in the standard form of section \ref{alg}. 

By our previous work in \cite{AlS} the space $\F_3=\star \F_4$ of infinitesimal deformations of this nearly $G_2$ structure is canonically identified to the 
Lie algebra $\su(3)$ by means of an explicit homomorphism $A \in \Hom_\h(\su(3), \Lambda^3_{27} \m)$ as follows. Map 
 $\xi \in \su(3)$ to a $\Lambda^3_{27} \m$ valued function on $G$ via $g\mapsto A(g^{-1} \xi g)$. Since the  $\SU(2)$-part of $G$ acts trivially on $\su(3)$ this actually lives on $SU(3)$. 
The well-known identification 
$
\Omega^3_{27}(G/H) = \Cc^\infty(G, \Lambda^3_{27} \m)^H
$ produces then an isomorphism 
$$\xi \in \su(3) \mapsto \beta_{\xi} \in \F_3.$$
W.r.t. this parametrisation the obstruction map $\mathbb{K}:\F_4 \to \F_4^{\star}$ from Theorem \ref{main12} reads 
\begin{equation} \label{intf}
 \mathbb{K}(\star \beta_{\xi})\star \beta_{\xi}=\frac{1}{\vol(U(1))}\int_{SU(3)}P(g\xi g^{-1}) \vol
\end{equation}
where the  polynomial $P:\su(3) \to \mathbb{R}$ is given by 
$
P(\xi ) = \la p(A(\xi), A(\xi)), i^{-1}(A(\xi))\ra
$
and the symmetric bilinear map $p:\Lambda^3_{27}\m \times \Lambda^3_{27}\m \to \Sym^2(\m)$ is defined according to 
$$ 
p(\gamma_1,\gamma_2)(v_1,v_2):=g(v_1 \lrcorner \gamma_1, v_2 \lrcorner \gamma_2).
$$
This follows in an essentially algebraic way from Proposition \ref{pol} and Remark \ref{rmks1},(ii).

The aim here is to show that all infinitesimal deformations in $\F_4$
are obstructed to second order. By an elementary harmonic analysis observation \cite{Fos} only the $SU(3)$-invariant piece in the polynomial 
$P$ can contribute to the integral \eqref{intf}. In turn this is determined by the scalar product of $P$ with the unique invariant  cubic polynomial on $\su(3)$ given by $\xi \mapsto i\det \xi $. 

In the next section we will show that this scalar product is non-vanishing by first computing the polynomial $P$. Equivalently 
$\mathbb{K}^{-1}(0)= \{0 \}$ thus completing the proof of Theorem \ref{main1}.

\medskip
\subsection{Computation of the obstruction polynomial} \label{obsp}
In this section all computations are performed on $\m$. It is convenient to write the $G_2$ form  $\varphi$ and its Hodge dual $\psi$ on $\m$ as
\begin{equation*}
\varphi=\vol_3 \,+\,  e^a \w \omega_a, \quad \psi=\vol_4  - \, (e^{12} \w \omega_3+e^{23} \w \omega_1+e^{31} \w \omega_2)
\end{equation*}
where we use the notation $\vol=e^{1234567}, \;   \vol_3 = e^{123}, \; \vol_4 = e^{4567}$, as well as 
$$
\omega_1=e^{45}-e^{67}, \ \omega_2=e^{46}+e^{57}, \ \omega_3=e^{47}-e^{56} \ .
$$
For $a=1,2,3$ we denote with $I_a$ the skew-symmetric endomorphisms  on $\m_4$ associated to the $2$-form $\omega_a$ via  $\omega_a(\cdot, \cdot ) = g(I_a \cdot, \cdot)$.
On basis elements these are determined from 
$I_1e_4=e_5, $ $\ I_1e_6=-e_7,$ $\ I_2e_4=e_6,$ $\ I_2e_5=e_7, 
I_3e_4=e_7$\, and $ I_3e_5=-e_6
$
and satisfy the relations $I_1I_2=-I_2I_1=-I_3$. Moreover the forms $\omega_a$ are anti-selfdual with $\omega_a\w \omega_b = - 2\delta_{ab} \vol_4$. 

\medskip

To outline how the map $A:\su(3) \to \Lambda^3_{27}\m $ is explicitly build pick $\xi \in \su(3)$ and write 
$$
\xi 
\,=\, 
\left (
\begin{array}{ccc} iv_1 &  x_1+ix_2& x_3+ix_4\\
-x_1+ix_2 & iv_2 & x_5+ix_6 \\
-x_3+ix_4 & -x_5+ix_6 & iv_3 \end{array} 
\right )
\,=\,
\left (
\begin{array}{ccc} iv_1 &- \bar z_3 & z_2\\
z_3 & iv_2 & -\bar z_1 \\
-\bar z_2 & z_1 & iv_3 \end{array} 
\right )
$$ 
with $v_1+v_2+v_3=0$ and $\, z_1= -x_5+ix_6, \,  z_2= x_3+ix_4, \,  z_3=-x_1+ix_2$. Following the detailed description in \cite{AlS} and using the matrix form of the basis 
elements $\{e_k, 1 \leq i \leq 7\}$ therein it follows that the map $A$
then breaks into three pieces 
\begin{equation*}
A(\xi) \,=s \tilde{\varphi} \; -\; \tfrac{5}{3}y \w \Omega \; +\, \tfrac{\sqrt{5}}{3\sqrt{2}}C(x).
\end{equation*}
Here $ \;\tilde{\varphi}:=\varphi-7\vol_3 \in  (\Lambda^1\m_3 \otimes \Lambda^2\m_4) \oplus \Lambda^3\m_3$, \; $\Omega:=e^{45}+e^{67} \in \Lambda^2\m_4$ and 
\begin{equation*}
\begin{split}
&y \,:=\,  \tfrac{v_1-v_2}{2}e_1-x_1e_2+x_2e_3\:   \in \, \m_3 \quad \\
&x \,:=\, x_3e_5-x_4e_4+x_5e_7-x_6e_6  \;\in\, \m_4 \\
&s:=\frac{v_1+v_2}{2}.
\end{split}
\end{equation*}
The last component in $A(\xi)$ belongs to $\Lambda^3\m_4\, \oplus \,(\Lambda^2\m_3 \otimes \Lambda^2 \m_4)$ and reads 
\begin{equation} \label{par1}
C(x) \,=\, 3x \lrcorner \vol_4 \,+\,e^{12} \w (x \lrcorner \omega_3)+e^{23} \w (x \lrcorner \omega_1)+e^{31} \wedge (x \lrcorner \omega_2)
\,=\, x \lrcorner (4\vol_4-\psi).
\end{equation}

\medskip

In order to streamline the computations below we will write $I_y=y_aI_a$ whenever $y=y_a e_a \in \m_3$. The symmetric tensor product on vectors is defined according to the convention   
$v \odot w=\tfrac12(v \otimes w+w \otimes v)$. Finally we denote with $J$  the complex structure on $\m_4$ 
defined via $g(J \cdot , \cdot ) = \Omega(\cdot , \cdot )$.  The restriction of the quadratic map $p$ to the subspace 
$\spann \{A(\xi) : \xi \in \su(3) \} \subseteq \Lambda^3_{27}\m$ 
is fully determined as follows. 
\begin{lema} \label{Lp}
The components of  $p(A(\xi),A(\xi))$ are given by\\[-.5ex]
\begin{equation*}
\begin{split}
&p(\tilde{\varphi},C(x)) = -4I_a x    \odot e_a \;\in\, \m_3 \odot \m_4  \\
&p(C(x),C(x)) = 2\vert x \vert^2\id_3 \;+\; 10(\,\vert x \vert^2\id_4 - x \otimes x)  \; \in\, \mathbb{R}\id_3 \oplus \Sym^2(\m_4) \\
&p(y \w \Omega,C(x)) = 6 y \odot Jx \in \m_3 \odot \m_4 \\
&p(\widetilde{\varphi},y \w \Omega) =   -JI_y    \;  \in\, \Sym^2_0\m_4 \\
&p(\widetilde{\varphi},\widetilde{\varphi}) =  38\id_3+3\id_{4}.
\end{split}
\end{equation*}
\end{lema}
The proof is by direct calculation using only the behaviour of the $3$-forms of type $\widetilde{\varphi}, y \w \Omega$ and 
$C(x)$ with respect to $\m=\m_3 \oplus \m_4$. The last piece of information needed in order to conclude is 
\begin{lema} \label{Lj}
The three endomorphisms corresponding to the summands of $A(\xi)$ are\\[-0.5ex]
\begin{equation*}
\begin{split}
&i^{-1}(C(x))  \;=\;  - \tfrac12   e_a \odot I_a x  \quad \in \m_3 \odot \m_4 \\
& i^{-1}(\tilde{\varphi})  \;=\;  -2\id_3+\tfrac{3}{2}\id_4\\    
& i^{-1}(y \w \Omega) \;=\; -\tfrac{1}{2}JI_y \quad \in \Sym_0^2\m_4 .
\end{split}
\end{equation*}
\end{lema}
The routine, though lengthy proof is based on the transformation formula \eqref{s-star} which allows computing $i^{-1}$. Alternatively these facts can be checked by writing down the action of $i$ on decomposable elements in $\Sym^2_0\m$. 

Further on, we write $A(\underline{\xi})=s\widetilde{\varphi}+y \w \Omega+C(x)$ and compute all terms present in the polynomial $P(\underline{\xi})$
by taking the scalar product of the quantities computed in Lemma \ref{Lj} and \ref{Lp}. The scalar product on symmetric tensors in  $\Sym^2(\m)$  is defined here by $\langle S_1,S_2\rangle=\tr(S_1S_2)$. Pure type considerations 
w.r.t. the splitting $\Sym^2\m=\Sym^2 \m_3 \oplus \Sym^2 \m_4 \oplus (\m_3 \odot \m_4)$ show that 
\begin{equation*}
\begin{split}
&\langle p(y \wedge \Omega, y \wedge \Omega), i^{-1}(A(\underline{\xi}))\rangle  \;=\; 2s\vert y\vert^2  \\
&\langle p(y \wedge \Omega, C(x)),  i^{-1}(A(\underline{\xi}))\rangle \;=\;- \tfrac32g(Jx,I_yx)  \\
& \langle p(C(x), C(x)),  i^{-1}(A(\underline{\xi}))\rangle \;=\; 33s\vert x\vert^2-5g(Jx,I_yx) \\
&\langle p(\widetilde{\varphi}, \widetilde{\varphi}),  i^{-1}(A(\underline{\xi}))\rangle \;=\; -210 s   \\
&\langle p(\widetilde{\varphi},y \wedge \Omega),  i^{-1}(A(\underline{\xi}))\rangle  \; =\; \quad 2\vert y \vert^2  \\
& \langle p(\widetilde{\varphi}, C(x)),  i^{-1}(A(\underline{\xi}))\rangle \;=  \quad 3 \vert x \vert^2.
\end{split}
\end{equation*}
\medskip
Plugging these relations into the expression for $P(\underline{\xi})$ leads to 
\begin{equation*}
\langle p(A(\underline{\xi}), A(\underline{\xi})), i^{-1}(A(\underline{\xi}))\rangle \;=\; -210s^3  \,+\, s(39\vert x \vert^2  \,+\,   6\vert y \vert^2 ) \,-\, 8 R(\xi)
\end{equation*}
where  $R(\xi) := g(Jx,I_yx)$. Reverting to the original parametrisation of $\xi$ via the transformations $y \mapsto -\frac{5y}{3}$ and $x \mapsto \frac{\sqrt{5}}{3\sqrt{2}}x$ yields 
\begin{equation*}
P(\xi) \;=\; 210s^3  \,+\,   \tfrac{65}{6} s \vert x \vert^2  \,+\,   \tfrac{50}{3} s \vert y \vert^2 \,+\,
 \tfrac{100}{27}R(\xi).
\end{equation*}

This expression can be refined as follows in order to facilitate the computation of $\langle P, i\det \rangle$. In the current complex notation further calculation shows  that $R(\xi)$ reads
\bea
R(\xi ) &=&  \tfrac{v_1-v_2}{2}(x_3^2 + x_4^2 - x_5^2 - x_6^2) - 2x_1( -x_3x_6 +x_4x_5)  + 2x_2( x_3x_5 +x_4x_6) \\[.5ex]
&=&  \tfrac{v_1-v_2}{2} (|z_2|^2 - |z_1|^2) \,+\,\, 2 i \Re(z_1z_2z_3) .
\eea

Similarly, we also  have
$$
|y|^2 \;=\; \tfrac{(v_1 - v_2)^2}{4} \,+\, |z_3|^2  \qquad\mbox{and}\qquad |x|^2 \;=\; |z_1|^2 \,+\, |z_2|^2 \ 
$$
as well as 
$$
i\det(\xi)\; =\; v_1v_2v_3 \,+\, 2i\Re(z_1z_2z_3) \,-\,  (v_1|z_1|^2 + v_2|z_2|^2 \,+\, v_3|z_3|^2) .
$$

\medskip

As explained above there remains to compute the scalar product of the two cubic polynomials $P$ and $i\det$ considered as elements in $ \Sym^3(\su(3))$.
We interpret the coordinates $v_1, v_2 , v_3=-v_1-v_2, z_1,z_2,z_3$ as linear forms on $\su(3)$. The scalar product in use on $\su(3)^*$ is induced by the trace form $b(\xi, \xi) = -\tfrac12 \tr(\xi^2)$. It is completely determined from
$
\la v_a, v_a \ra = \tfrac43, \la v_a, v_b \ra = -\tfrac23
$
for $a\neq b$ and
$
\la z_j, \bar z_k\ra = 2 \delta_{jk}
$.
Then the scalar product on monomials, i.e. elements of $ \Sym^3(\su(3)^\ast)$, is computed with the help of  permanents.
Recall that the permanent of a matrix $A=(a_{ij})$ is defined similarly to the determinant by the formula 
$\, \mathrm{perm}\,A = \sum_{\sigma \in S_n} \prod_{i} a_{i,\sigma(i)}$. Permanents help make explicit the scalar product on 
$\Sym^k(\su(3)^\ast)$ induced from the scalar product $\la \cdot, \cdot \ra$ on $\su(3)^\ast$ via 
$
\la a_1 \ldots  a_k, b_1 \ldots b_k \ra  = \mathrm{perm}\, (\la a_i, b_j\ra)
$, 
where $a_1 \ldots  a_k$ and $b_1 \ldots b_k$ are monomials in the coordinate functions. 
A straightforward  calculations then gives
$$
\la s^3, i \det  \ra = -\tfrac49, \;  \; \la s |x|^2 , i \det  \ra = -\tfrac83 , \;      \;    \la s|y|^2, i \det  \ra = 4   , \; \;\la R, i \det \ra = 24 \ .
$$
Collecting all summands we finally obtain
$$
\la P,  i \det \ra \;=\; 210 (-\tfrac49)  \,+\, \tfrac{65}{6} (-\tfrac83)  \, +  \, \tfrac{50}{3} 4  \, + \, \tfrac{100}{27} 24 \;=\;  \tfrac{100}{3} \:\neq\; 0 .
 $$
This proves Theorem \ref{main1}.

\medskip
{\it{Acknowledgements.}} It is a pleasure to thank G.Weingart for his input on some of the computations in the last section.


\begin{thebibliography}{99}
\bibitem{AlS}
B.Alexandrov, U.Semmelmann, 
\textit{Deformations of nearly parallel $G_2$-structures}, 
Asian J. Math. {\bf{16}} (2012), 713--744. 

\bibitem{BO}
G.Ball, G.Oliveira, 
\textit{Gauge theory on Aloff-Wallach spaces},
Geom. Topol. {\bf{23}} (2019), no. 2, 685--743.

\bibitem{Besse}
A.Besse, \textit{Einstein manifolds}, Springer-Verlag, Berlin, 2008.

\bibitem{Br}
R. Bryant, 
\textit{Some remarks on $G_2$-structures}, Proceedings of G\"okova Geometry-Topology Conference 2005, pp. 75-109, G\"okova Geometry/Topology Conference (GGT), G\"okova, 2006.



\bibitem{CI}
R.Cleyton, S.Ivanov, 
\textit{On the geometry of closed $G_2$-structures},
Comm. Math. Phys. {\bf{270}} (2007), no. 1, 53--67.

\bibitem{BFK}
H.Baum, Baum, Th.Friedrich, R.Grunewald, I.Kath, \textit{
Twistors and Killing spinors on Riemannian manifolds}, Teubner Texts in Mathematics {\bf{124}}, 1991.

\bibitem{Ca}
F.M.Cabrera, M.D.Monar, A.Swann, \textit{
Classification of $G_2$-structures}, J. London Math. Soc. (2) {\bf{53}} (1996), no. 2, 407--416.

\bibitem{FKMS}
Th.Friedrich, I.Kath, A.Moroianu, U.Semmelmann,
\textit{On nearly parallel $G_2$-structures},
J. Geom. Phys. {\bf{23}} (1997), no. 3-4, 259--286.

\bibitem{Fos}
L.Foscolo, 
\textit{Deformation theory of nearly K\"ahler manifolds},
J. Lond. Math. Soc. (2) {\bf{95}} (2017), no. 2, 586--612.

\bibitem{GS}
K.Galicki, S.Salamon, \textit{Betti numbers of $3$-Sasakian manifolds}, 
Geom. Dedicata {\bf{63}} (1996), no.1, 45--68.

\bibitem{Hi}
N.Hitchin, \textit{Stable forms and special metrics}, {\it{Global differential geometry: the mathematical legacy of Alfred Gray (Bilbao, 2000)}}, 70--89,
Contemp. Math., {\bf{288}}, Amer. Math. Soc., Providence, RI, 2001.


\bibitem{Coev1}
C.van Coevering, C.Tipler, 
\textit{Deformations of constant scalar curvature Sasakian metrics and K-stability},
Int. Math. Res. Not. IMRN 2015, no. 22, 11566--11604.

\bibitem{coev2}
C.van Coevering, 
\textit{Deformations of Killing spinors on Sasakian and $3$-Sasakian manifolds},
J. Math. Soc. Japan {\bf{69}} (2017), no. 1, 53--91.

\bibitem{MNS}
A.Moroianu, P.-A.Nagy, U.Semmelmann, \textit{
Deformations of nearly K\"ahler structures}, Pacific J. Math. {\bf{235}} (2008), no. 1, 57-72.

\bibitem{MS}
A.Moroianu, U.Semmelmann, \textit{
Infinitesimal Einstein deformations of nearly K\"ahler metrics},
Trans. Amer. Math. Soc. {\bf{363}} (2011), no. 6, 3057--3069.
 \end{thebibliography}
\end{document}